\documentclass[12pt]{amsart}

\usepackage{amsmath}
\usepackage{amssymb}
\usepackage{amsthm}
\usepackage{xcolor}
\usepackage{bm}
\usepackage{graphicx}
\usepackage{pifont}
\usepackage{mathrsfs}
\usepackage{epsfig,verbatim}
\usepackage[pagewise]{lineno}\linenumbers

\newcommand\be{\begin{equation}}
\newcommand\ee{\end{equation}}
\newcommand\bea{\begin{eqnarray}}
\newcommand\eea{\end{eqnarray}}
\newcommand\beaa{\begin{eqnarray*}}
\newcommand\eeaa{\end{eqnarray*}}
\newcommand\beba{\begin{equation}\left\{\begin{array}{rcl}}
\newcommand\eeba{\end{array}\right.\end{equation}}
\newcommand\bebaa{\begin{equation*}\left\{\begin{array}{rcl}}
\newcommand\eebaa{\end{array}\right.\end{equation*}}

\newcommand\red{\color{red}}

\newcommand\bR{{\mathbb{R}}}

\newcommand\bN{{\mathbb{N}}}

\newcommand\uphi{{\overline{\phi}}}
\newcommand\lphi{{\underline{\phi}}}
\newcommand\upsi{{\overline{\psi}}}
\newcommand\lpsi{{\underline{\psi}}}

\newcommand{\ep}{\epsilon}

\newcommand{\dint}{\displaystyle \int}
\newcommand{\dsum}{\displaystyle \sum}

\newcommand{\dlim}{\displaystyle \lim}

\newtheorem{theorem}{Theorem}[section]
\newtheorem{lemma}[theorem]{Lemma}
\newtheorem{corollary}[theorem]{Corollary}

\newtheorem{definition}[theorem]{Definition}
\newtheorem{remark}[theorem]{Remark}
\newtheorem{proposition}[theorem]{Proposition}

\numberwithin{equation}{section}

\begin{document}
\nolinenumbers

\title[Propagating front solutions]{Propagating front solutions in a time-fractional Fisher-KPP equation}

\author[H. Ishii]{Hiroshi Ishii}
\address{Research Center of Mathematics for Social Creativity, Research Institute for Electronic Science, Hokkaido University, Hokkaido, 060-0812, Japan}
\email{hiroshi.ishii@es.hokudai.ac.jp}

\thanks{Date: \today. Corresponding author: H. Ishii}


\begin{abstract}
In this paper, we treat the Fisher-KPP equation with a Caputo-type time fractional derivative and discuss the propagation speed of the solution.
The equation is a mathematical model that describes the processes of sub-diffusion, proliferation, and saturation.
We first consider a traveling wave solution in studying the propagation of the solution,
but we cannot define it in the usual sense, owing to the time fractional derivative in the equation.
We therefore assume that the solution asymptotically approaches a traveling wave solution,
and the asymptotic traveling wave solution is formally introduced as a potential asymptotic form of the
solution.
The existence and properties of the asymptotic traveling wave solution are discussed using a monotone iteration method.
Finally, the behavior of the solution is analyzed by conducting numerical simulations based on the results for asymptotic traveling wave solutions.
\end{abstract}

\maketitle
\setcounter{tocdepth}{3}

\section{Introduction}
In this paper, we consider the propagation of solutions to a time-fractional evolution equation
\be\label{eq:main}
\partial^{\alpha}_{t} u = u_{xx}+ f(u) \quad (t>0,\ x\in\bR),
\ee
where $u=u(t,x)\in\bR\ (t>0,\ x\in\bR)$, and $f:\bR\to\bR$ is
a monostable nonlinearity
\beaa
f\in C^1(\bR),\quad f(0)=f(1)=0,\quad f(u)>0\quad (0<u<1),
\eeaa
satisfying the KPP-type condition (named after Kolmogorov, Petrovsky, and Piskunov \cite{KPP})
\beaa
f(u)\le f'(0)u\ (0<u<1).
\eeaa
One example is $f(u)=u(1-u)$.
Further, $\partial^{\alpha}_{t}$ is a Caputo derivative defined as
\beaa
\partial^{\alpha}_{t} u (t,x):= \dfrac{1}{\Gamma(1-\alpha)} \int^{t}_{0} \dfrac{\partial u}{\partial s}(s,x) \dfrac{ds}{(t-s)^{\alpha}},
\eeaa
where $\alpha\in (0,1)$ and $\Gamma$ denotes the gamma function


\medskip
The equation with $\alpha=1$ corresponds to the Fisher-KPP equation
\be\label{eq:Fisher}
u_t = u_{xx} + f(u)\quad (t>0,\ x\in\bR),
\ee
which is a well-known mathematical model used in the study of population dynamics \cite{Fisher, Fisher2, KPP}.
A typical structure of propagation phenomena is the traveling wave solution.
In the general case, a solution $u$ is called a traveling wave solution if there exist a speed $c>0$ and a wave profile $\phi\in C(\bR)$ such that $u(t,x)=\phi(x+ct)$.
In the case of the Fisher-KPP equation,
by substituting this expression into the equation and introducing the moving coordinate $\xi:=x+ct$,
$(c,\phi(\xi))$ must satisfy
\beaa
c\phi'(\xi) = \phi''(\xi) + f(\phi(\xi))\quad (\xi\in\bR).
\eeaa
As a pioneering work, Fisher \cite{Fisher2} considered the case that $f(u)=u(1-u)$ and found that there is $c^{*}_1>0$ such that \eqref{eq:Fisher} has a monotone traveling wave solution connecting from $0$ to $1$ when $c\ge c^{*}_1$, and no such solution when $c\in (0,c^{*}_1)$.
In addition, Kolmogorov, Petrovsky, and Piskunov \cite{KPP} revealed the existence of a traveling wave solution and its long-time behavior under more general conditions for $f(u)$.
Building on this work, studies have expanded into various areas, such as traveling wave solutions (e.g. \cite{Ma}),
long-time behavior (e.g. \cite{AW}), and entire solutions (e.g. \cite{HN}).
Here, we note that $c^{*}_1=2\sqrt{f'(0)}$ holds in the case of KPP-type nonlinearity,
and it is known that the propagation of the solution proceeds with the asymptotic highest term $c^{*}_1 t$ under suitable conditions \cite{AW,KPP}.


\medskip
Meanwhile, research on different diffusion processes has expanded into various contexts.
One extension is a spatially nonlocal model of diffusion
\be\label{eq:non}
u_t = D[u] + f(u)\quad (t>0,\ x\in\bR).
\ee
Examples of the diffusion term $D[u]$ are the fractional Laplacian \cite{BV}:
\beaa
D[u] = -(-\Delta)^{s} u\quad (0<s<1)
\eeaa
and nonlocal diffusion \cite{AMRT, HMMV}:
\beaa
D[u] = \int_{\bR} K(x-y)(u(t,y)-u(t,x))dy\quad (K\in L^1(\bR)).
\eeaa
In the case of nonlocal diffusion with a rapidly decaying kernel $K(x)$,
it has been established that results of traveling wave solutions similar to \eqref{eq:Fisher} can be obtained \cite{CG, CD, EGIW, SLW, Ya}.
In the case of the fractional Laplacian and nonlocal diffusion with slowly decaying kernels,
there has been widespread analysis of traveling wave solutions and solution propagation \cite{BCC, BFRW, GZ, VNK, VNN-F}.
In such cases with a monostable nonlinearity such as $f(u)=u(1-u)$,
it is known that \eqref{eq:non} has no traveling wave solutions \cite{CR, Ya}.
Moreover, the long-time behavior of the solution has been studied,
revealing that the solution propagates more rapidly than $O(t)$ \cite{CR, Garnier}.
As it is not trivial that solution propagation occurs at $O(t)$ as the cited studies show, it is sometimes questionable whether it is appropriate to consider traveling wave solutions.


\medskip
Returning to \eqref{eq:main},
the diffusion equation is modified to include a time-fractional derivative,
\beaa
\partial^{\alpha}_t u = u_{xx}\quad (t>0,\ x\in\bR),
\eeaa
which is called the diffusion-wave equation or the time-fractional diffusion equation.
The fundamental solution $G_{\alpha}(t,x)$ is defined by
\beaa
G_{\alpha}(t,x) := \dfrac{1}{\sqrt{\pi} |x| } H^{2,0}_{1,2}
\left[ \dfrac{x^2}{4t^{\alpha}} \middle|
\begin{matrix}
 \left( 1, \alpha \right) & \\
(1/2,1)  & \left( 1 , 1 \right)
\end{matrix}
\right],
\eeaa
where $H^{\mu,\nu}_{p,q}(z)$ is Fox's H-function \cite{EK}.
According to the properties of Fox's H function,
the spatial decay rate is given by
\beaa
G_{\alpha}\left(1,\dfrac{2x}{\alpha}\right) \simeq C_{\alpha} x^{(\alpha-1)/(2-\alpha)} \exp\left( -\dfrac{2-\alpha}{\alpha} x^{2/(2-\alpha)} \right)
\eeaa
(see \cite{EK, Mainardi}).
It is known that the diffusion is slower than local diffusion \cite{EK, KRY}.
In fact, as $G_{\alpha}$ satisfies $G_{\alpha}(t,x) = t^{-\alpha/2} G_{\alpha}(1, t^{-\alpha/2}x)$,
the variance of the fundamental solution is
\beaa
\int_{\bR} x^2 G_{\alpha}(t,x)dx = \left( \int_{\bR} x^2 G_{\alpha}(1,t^{-\alpha/2}x)dx \right) t^{-\alpha/2}  = \left(\int_{\bR} x^2 G_{\alpha}(1,x)dx\right) t^{\alpha}.
\eeaa
Therefore, the variance grows at a rate slower than $O(t)$ for the usual diffusion equation for sufficiently large times $t>0$.
In addition,  it becomes slower as $\alpha$ decreases.
To analyze phenomena with such sub-diffusion and nonlinearity,
Equation \eqref{eq:main} and similar models have been proposed (e.g., \cite{Fedotov, FSSS, FSSS2, HLW, VNN}). In particular, model \eqref{eq:main} was proposed by \cite{HLW}
as a model of anomalous diffusion with linear reaction dynamics.
In our context with $f(u)=u(1-u)$, model \eqref{eq:main} is a natural extension of their model
with a newly added saturation effect.


\medskip
In analyzing the solution to equation \eqref{eq:main}, we should keep in mind that the dynamical properties change owing to the time-fractional derivative.
As an example, let us consider the time-fractional Malthus model as the simplest model:
\beaa
\partial^{\alpha}_t v = \zeta v \quad (t>0),
\eeaa
where $v=v(t)\in\bR$, and $\zeta$ is a non-zero constant.
The solution is given by
\beaa
v(t) = v(0) E_{\alpha,1}( \zeta  t^{\alpha}),
\eeaa
where $E_{\alpha,\beta}(z):= \dsum^{\infty}_{k=0}\dfrac{z^{k}}{\Gamma(\alpha k + \beta)}$ is the Mittag--Leffler function \cite{HMS, KRY}.
Moreover, the asymptotic behavior of $E_{\alpha,1}$ is
\be\label{mittag}
E_{\alpha,1}(z) \simeq {\red
\begin{cases}
    \dfrac{1}{\alpha} \exp(z^{1/\alpha})  &(z\to+\infty),\\
    - \dfrac{1}{z \Gamma(1-\alpha)} &(z\to-\infty)
\end{cases}}
\ee
(see Section 6 in \cite{HMS}).
Hence, $v(t)$ converges to $0$ with polynomial order when $\zeta<0$, and grows exponentially when $\zeta>0$.
These properties often appear in time-fractional nonlinear equations when analyzing the decay and shape of characteristic solutions (e.g. \cite{KRY, VNN}).

\medskip
In this paper, we consider propagating solutions of the front type,
such as the traveling wave solutions we have been discussing so far.
Long-time behavior has been considered in various settings for propagation and saturation processes
under slow diffusion \cite{Fedotov, FM, FSSS, FSSS2}.
However, to the best of our knowledge, no characterization of propagation solutions based on the discussion of traveling wave solutions has been done for equation \eqref{eq:main} with a monostable nonlinearity.
Therefore, this paper considers the relationship between the solution behavior and $\alpha$
by introducing the notion of a suitable solution like a traveling wave solution.


\medskip
The paper is organized as follows.
In Section 2, we introduce the concept of the traveling wave structure, termed the asymptotic traveling wave solution, and present the main result regarding the existence.
The proof is given in Section 3.
We then present numerical simulations under appropriate conditions in Section 4.
Finally, we highlight unresolved issues in Section 5.

\bigskip
\section{Main results}

In analyzing the propagation of the solution, the concept of a traveling wave solution is considered.
However, the usual method of obtaining traveling wave solutions cannot be immediately applied in the case of time-fractional equations.
For $c>0$ and $\phi\in C^2(\bR)$, the Caputo derivative of a traveling wave solution $u(t,x)=\phi(x+ct)$ is
\beaa
\partial^{\alpha}_{t} u(t,x) &=& \dfrac{c}{\Gamma(1-\alpha)} \int^{t}_{0} \dfrac{\phi'(x+cs)}{(t-s)^{\alpha}} ds \\
&=& \dfrac{c}{\Gamma(1-\alpha)} \int^{t}_{0} \dfrac{\phi'(x+c(t-s))}{s^{\alpha}} ds \\
&=& \dfrac{c^{\alpha}}{\Gamma(1-\alpha)} \int^{ct}_{0} \dfrac{\phi'(x+ct-s)}{s^{\alpha}} ds.
\eeaa
We recall the moving coordinate $\xi=x+ct$ to obtain
\be\label{eq:Caputo}
\partial^{\alpha}_{t} u(t,x)= \dfrac{c^{\alpha}}{\Gamma(1-\alpha)} \int^{ct}_{0} \dfrac{\phi'(\xi-s)}{s^{\alpha}} ds.
\ee
This means that the term coming out of the time derivative depends on $t$.
Hence, it is impossible to define a traveling wave solution in the classical sense.

On the other hand, in \cite{VNN}, they considered traveling wave solutions whose solutions might be approached asymptotically  to investigate the final state of a front-type solution that connects different stable states.
Based on this idea, we define an asymptotic traveling wave solution to analyze the final state.
That is to say, the asymptotic traveling wave solution is given in the following form by setting $t\to+\infty$ in \eqref{eq:Caputo}.
\begin{definition}
$(c,\phi)\in(0,+\infty)\times C^2(\bR)$ is an {\bf asymptotic traveling wave solution} of \eqref{eq:main} if it satisfies
\be\label{eq:tw}
\dfrac{c^{\alpha}}{\Gamma(1-\alpha)} \int^{+\infty}_{0} \dfrac{\phi'(\xi-s)}{s^{\alpha}} ds = \phi''(\xi) + f(\phi (\xi))
\ee
for all $\xi\in\bR$.
\end{definition}
In particular, we focus on a front-type asymptotic traveling wave solution connecting two constant states:
\be\label{boundary}
\phi(-\infty)=0,\quad \phi(+\infty)=1.
\ee

Studies relating to equation \eqref{eq:tw} have been reported.
In \cite{NVN}, the bistable nonlinear term $f(u)$, which is piecewise linear, was considered,
and an attempt was made to analyze $(c,\phi)$ by describing its solution connecting two stable states explicitly.
Furthermore, in \cite{ACH, CA}, traveling wave solutions in a nonlocal Korteweg–de Vries–Burgers equation,
which are analogous to equation \eqref{eq:tw},
were considered,
and their existence was explored by investigating the properties of the linear equation and constructing exponentially decaying solutions.
In the case of the monostable nonlinear term $f(u)$ that we deal with,
there are no results on the existence of solutions to equation \eqref{eq:tw} to the best of our knowledge.

Thus, first consider the existence of the asymptotic traveling wave solutions.
Before stating the main theorem, we give the condition for $f(u)$ as
\be\label{ass:non}
\begin{cases}
f\in C^{1}(\bR),\ f(0)=f(1)=0, \\
\exists M>0,\ \exists a>0\ \mathrm{s.t.}\ -Mu^{1+a} \le f(u) - f'(0) u \le 0\ (0\le u\le 1).
\end{cases}
\ee
The condition includes the case that $f(u)=u(1-u)$.
We then obtain the following theorem.

\begin{theorem}\label{thm:main}
Let $\alpha\in(0,1)$ and $c^{*}_{\alpha}:= \dfrac{2^{1/\alpha}}{\sqrt{\alpha}} \left( \dfrac{f'(0)}{2-\alpha} \right)^{(2-\alpha)/2\alpha}$.
Then, for any $c\ge c^{*}_{\alpha}$ there exists an asymptotic traveling wave solution of \eqref{eq:main} with an increasing function $\phi$ satisfying \eqref{boundary}.
Furthermore, for $c>c^{*}_{\alpha}$, it can be normalized by
\beaa
\dlim_{\xi\to-\infty} e^{-\lambda_1\xi} \phi(\xi) = 1,
\eeaa
where $\lambda_1$ is the minimal positive root of $\lambda^2 - (c\lambda)^{\alpha}+f'(0)=0$.
\end{theorem}

\medskip
We have established the monotonicity of asymptotic traveling wave solutions and the exponential decay as $\xi\to-\infty$.
Whether $c^{*}_{\alpha}$ is the minimum speed remains unresolved
and is a problem to be considered in future work.
The proof is given in Section \ref{sec:proof}.

Owing to the monotonicity, we obtain the correspondence of asymptotic traveling wave solutions to the time evolution equation \eqref{eq:main}.
\begin{corollary}
For $\alpha\in(0,1)$ and $c\ge c^{*}_{\alpha}$, we set $v(t,x)=\phi(x+ct)$, where $\phi$ is constructed in Theorem \ref{thm:main}.
Then, $v(t,x)$ is a sub-solution of \eqref{eq:main}, that is, $v(t,x)$ satisfies
\beaa
\partial^{\alpha}_t v < v_{xx} + f(v) \quad (t>0,\ x\in\bR).
\eeaa
\end{corollary}
\begin{proof}
Fix $\alpha\in (0,1)$ and $c\ge c^{*}_{\alpha}$.
Then, from \eqref{eq:Caputo} and the definition of the asymptotic traveling wave solution, we have
\beaa
\partial^{\alpha}_t v - (v_{xx} + f(v)) = -\dfrac{c^{\alpha}}{\Gamma(1-\alpha)}\int^{+\infty}_{ct} \dfrac{\phi'(\xi-s)}{s^{\alpha}}ds <0
\eeaa
for all $t>0$ and $x\in\bR$.
\end{proof}

\medskip
According to the comparison principle,
solutions to Equation \eqref{eq:main} are above the asymptotic traveling wave solution if the initial states are properly given.
An example of an initial state is an upper solution constructed in Lemma \ref{lem:up}.
See \cite{FLY, KRY, LY} for results on comparison principles in fractional differential equations.

\begin{remark}
By taking the limits, we obtain
\beaa
\lim_{\alpha\to 1} c^{*}_{\alpha} = 2\sqrt{f'(0)}= c^{*}_1,\quad
\dlim_{\alpha\to+0} c^{*}_{\alpha} =
\begin{cases}
+\infty &f'(0) \ge 1,\\
0 &f'(0)<1.
\end{cases}
\eeaa
$c^{*}_{\alpha}$ connects to the minimum speed when $\alpha=1$.
In addition, the behavior depends on the value of $f'(0)$ as $\alpha$ approaches zero.
\end{remark}

\bigskip
\section{Existence of asymptotic traveling wave solutions}\label{sec:proof}
In this section, we give a proof of Theorem \ref{thm:main}.
Equation \eqref{eq:tw} is nonlocal,
and it is not easy to perform an analysis that considers local properties, such as a phase-plane analysis,
to investigate the existence and properties of asymptotic traveling wave solutions.
Therefore, we apply a monotone iteration method based on the maximum principle.

Throughout this section, we suppose that $\alpha\in (0,1)$ and $c>0$.
For convenience, we define the fractional derivative
\beaa
\partial^{\alpha}_{\xi}\psi(\xi) := \dfrac{1}{\Gamma(1-\alpha)} \int^{+\infty}_{0} \dfrac{\psi'(\xi-s)}{s^{\alpha}} ds
\eeaa
and the set
\beaa
C^{k}_{b} (\bR) := \left\{ \psi\in C^k(\bR)\ \middle|\ \sup_{\xi\in\bR}|\psi^{(j)}(\xi)|<+\infty\quad (j=0,1,\ldots,k) \right\}.
\eeaa
Also, we denote $C_b(\bR):= C^{0}_b(\bR)$.

\medskip
\subsection{Preliminaries}
First, we analyze the existence of a solution to the linear problem
\be\label{eq:linear}
L \psi (\xi) + g(\xi)= 0
\ee
for $g\in C_b(\bR)$, where $L$ is a linear operator
\beaa
L\psi(\xi):= \psi''(\xi)- c^{\alpha}\partial^{\alpha}_{\xi} \psi(\xi) -\kappa^2 \psi(\xi)
\eeaa
where $\kappa$ is any positive constant satisfying $\kappa^2> \max_{u\in [0,1]} |f'(u)|$.
We will fix it later.
This problem corresponds to equation \eqref{eq:tw} if $g(\xi)= \kappa^2 \psi(\xi)+f(\psi(\xi))$.

To solve problem \eqref{eq:linear},
we derive the integral equation and apply the Neumann series argument for the existence of a solution.
Let us define the functions
\beaa
K_{0} (\xi) := \dfrac{1}{2\kappa} e^{-\kappa|\xi|},\quad
K_{\alpha}(\xi):= -c^{\alpha} \partial^{\alpha}_{\xi}K_{0}(\xi) = -\dfrac{c^{\alpha}}{\Gamma(1-\alpha)} \int^{+\infty}_{0} \dfrac{K'_0(\xi-s)}{s^{\alpha}}ds.
\eeaa
We note that for all $g\in C_b(\bR)$, $(K_0*g)$ satisfies
\beaa
(K_0*g)''(\xi) - \kappa^2(K_0*g)(\xi) + g(\xi)=0.
\eeaa
This implies $K_0*g$ belongs to $C^2_b(\bR)$ for any $g\in C_b(\bR)$.

We first check the properties of $K_{\alpha}$.
\begin{lemma}\label{lem:ker}
$K_{\alpha}$ has the following properties:
\begin{itemize}
    \item[(i)] $K_{\alpha}\in C(\bR)$;
    \item[(ii)] $K_{\alpha}(\xi)= -(c\kappa)^{\alpha} K_0(\xi)$ holds for $\xi\le 0$;
    \item[(iii)] $\dlim_{\xi\to+\infty} \xi^{1+\alpha} K_{\alpha}(\xi)= \dfrac{\alpha c^{\alpha}}{{\red\kappa^2}\Gamma(1-\alpha)}$;
\end{itemize}
\end{lemma}
\begin{proof}
(i) $K_{\alpha}(\xi)$ is represented by
\beaa
K_{\alpha}(\xi) = -\dfrac{c^{\alpha}}{{\red 2\Gamma(1-\alpha)}}\left\{ e^{\kappa\xi} \int^{+\infty}_{\max\{\xi,0\}} \dfrac{e^{-\kappa s}}{s^{\alpha}}ds -e^{-\kappa \xi} \int^{\max\{\xi,0\}}_{0} \dfrac{e^{\kappa s}}{s^{\alpha}} ds\right\}.
\eeaa
It is easy to see that both integrals are well-defined for all $\xi\in\bR$ and continuous with respect to $\xi$.

(ii) For $\xi\le 0$, we have
\beaa
K_{\alpha}(\xi)= -\dfrac{c^{\alpha} e^{\kappa\xi}}{{\red 2 \Gamma(1-\alpha)}} \int^{+\infty}_{0} \dfrac{e^{-\kappa s}}{s^{\alpha}}ds = -\dfrac{c^{\alpha} \kappa^{\alpha-1}}{2} e^{\kappa\xi} = -(c\kappa)^{\alpha} K_0(\xi).
\eeaa

(iii) When $\kappa\xi>1$, we obtain
\beaa
K_{\alpha}(\xi) &=& -\dfrac{c^{\alpha}}{2 \Gamma(1-\alpha)}\left\{ e^{\kappa\xi} \int^{+\infty}_{\xi} \dfrac{e^{-\kappa s}}{s^{\alpha}}ds -e^{-\kappa \xi} \int^{\xi}_{0} \dfrac{e^{\kappa s}}{s^{\alpha}} ds\right\} \\
&=& -\dfrac{(c\kappa)^{\alpha}}{2\kappa \Gamma(1-\alpha)}\left\{ e^{\kappa\xi} \int^{+\infty}_{\kappa\xi} \dfrac{e^{- s}}{s^{\alpha}}ds -e^{-\kappa \xi} \int^{\kappa\xi}_{0} \dfrac{e^{s}}{s^{\alpha}} ds\right\}.
\eeaa
We note that
\beaa
\int^{+\infty}_{\kappa\xi} \dfrac{e^{- s}}{s^{\alpha}}ds &=& \left[ -\dfrac{e^{-s}}{s^{\alpha}} \right]^{s=+\infty}_{s=\kappa \xi} - \alpha \int^{+\infty}_{\kappa\xi} \dfrac{e^{-s}}{s^{1+\alpha}} ds \\
&=& \dfrac{e^{-\kappa\xi}}{(\kappa\xi)^{\alpha}} - \alpha \int^{+\infty}_{\kappa\xi} \dfrac{e^{-s}}{s^{1+\alpha}} ds
\eeaa
and for $\delta\in (0,1)$,
\beaa
\int^{\kappa\xi}_{0} \dfrac{e^{s}}{s^{\alpha}} ds &=&  \int^{\delta}_{0} \dfrac{e^{s}}{s^{\alpha}} ds + \int^{\kappa\xi}_{\delta} \dfrac{e^{s}}{s^{\alpha}}ds \\
&=& \int^{\delta}_{0} \dfrac{e^{s}}{s^{\alpha}} ds +
\dfrac{e^{\kappa\xi}}{(\kappa\xi)^{\alpha}} - \dfrac{e^{\delta}}{\delta^{\alpha}}  + \alpha \int^{\kappa\xi}_{\delta} \dfrac{e^{s}}{s^{1+\alpha}}ds.
\eeaa
Then, we have
\beaa
&&e^{\kappa\xi} \int^{+\infty}_{\kappa\xi} \dfrac{e^{- s}}{s^{\alpha}}ds -e^{-\kappa \xi} \int^{\kappa\xi}_{0} \dfrac{e^{s}}{s^{\alpha}} ds \\
&&= -\alpha e^{\kappa\xi} \int^{+\infty}_{\kappa\xi} \dfrac{e^{- s}}{s^{1+\alpha}}ds -\alpha e^{-\kappa \xi} \int^{\kappa\xi}_{\delta} \dfrac{e^{s}}{s^{1+\alpha}} ds {\red -} \left( \int^{\delta}_{0}\dfrac{e^{s}}{s^{\alpha}}ds - \dfrac{e^{\delta}}{\delta^{\alpha}} \right) e^{-\kappa\xi}.
\eeaa
Hence, we obtain
\beaa
\dlim_{\xi\to+\infty} \xi^{1+\alpha} K_{\alpha}(\xi)= \dfrac{\alpha (c\kappa)^{\alpha}}{{\red \kappa}\Gamma(1-\alpha)} \kappa^{-(1+\alpha)} = \dfrac{\alpha c^{\alpha}}{{\red \kappa^2}\Gamma(1-\alpha)}.
\eeaa
from l'H\^{o}pital's rule.
\end{proof}

\medskip
Next, let us derive the integral equation.
For the derivation, we use the fact that $K_0$ is a Green's function of $(\kappa^2 -\partial^2_{x})$.
Then, we obtain the following integral equation.

\begin{lemma}\label{lem:integ}
Let $\psi\in C^2_b(\bR)$ be a solution of the integral equation
\be\label{eq:integ}
\psi(\xi) = (K_{\alpha}*\psi)(\xi) + K_0*g(\xi).
\ee
Then, $\psi$ is a solution to \eqref{eq:linear}.
\end{lemma}
\begin{proof}
We first show that
\beaa
c^{\alpha}(K_0 * \partial^{\alpha}_{\xi}\psi) = -(K_{\alpha}*\psi).
\eeaa
Changing variables $z=y-s$ and integrating by parts, we obtain
\beaa
c^{\alpha}(K_0 * \partial^{\alpha}_{\xi}\psi)(\xi)
&=& \dfrac{c^{\alpha}}{\Gamma(1-\alpha)} \int_{\bR} \int^{+\infty}_{0} K_0(\xi-y) \dfrac{\psi'(y-s)}{s^{\alpha}}dsdy \\
&=& \dfrac{c^{\alpha}}{\Gamma(1-\alpha)}\int^{+\infty}_{0} \int_{\bR} \dfrac{K_0(\xi-z-s)}{s^{\alpha}} \psi'(z) dz ds\\
&=& \dfrac{c^{\alpha}}{\Gamma(1-\alpha)}  \int^{+\infty}_{0} \int_{\bR} \dfrac{K'_0(\xi-z-s)}{s^{\alpha}} \psi(z)dz ds\\
&=& -(K_{\alpha}*\psi)(\xi)
\eeaa
for all $\xi\in\bR$.

Next, let us prove that $\psi$ is a solution to {\red \eqref{eq:linear}}.
From the properties of $K_0(\xi)$, we deduce
\beaa
\psi''(\xi) - \kappa^2\psi + g(\xi) &=& (K_{\alpha}*\psi)''(\xi) - \kappa^2 (K_{\alpha}*\psi)(\xi).
\eeaa
Then, we have
\beaa
&& (K_{\alpha}*\psi)''(\xi) - \kappa^2 (K_{\alpha}*\psi)(\xi) \\
&&\quad = -c^{\alpha}\left\{ (K_0 * \partial^{\alpha}_{\xi}\psi)''(\xi) - \kappa^2 (K_0 * \partial^{\alpha}_{\xi}\psi)(\xi) \right\} \\
&&\quad =c^{\alpha} \partial^{\alpha}_{\xi} \psi(\xi).
\eeaa
Thus, $\psi$ satisfies $L\psi(\xi)+g(\xi)=0$.
\end{proof}

\medskip
Let us show the existence of the solution to \eqref{eq:integ} in $L^{\infty}(\bR)$.
To apply the Neumann series argument,
we check the property of $\|K_{\alpha}\|_{L^1}$.

\begin{lemma}
There is $\Theta>0$ which is independent of $\kappa$ such that $\|K_{\alpha}\|_{L^1}< \Theta \kappa^{\alpha-2}$ holds for all $\kappa>1$.
\end{lemma}
\begin{proof}
To obtain the desired assertion, we evaluate $\|K_{\alpha}\|_{L^1}$ as
\beaa
\int_{\bR}|K_{\alpha}(\xi)|d\xi &=&
\left[ \int^{0}_{-\infty} + \int^{1/\kappa}_0 + \int^{+\infty}_{1/\kappa} \right]|K_{\alpha}(\xi)|d\xi.
\eeaa
Since we have (ii) in Lemma \ref{lem:ker},
the first integral part is computed by
\be\label{ine:neg}
\int^{0}_{-\infty} |K_{\alpha}(\xi)| d\xi = \dfrac{(c\kappa)^{\alpha}}{2\kappa^2}.
\ee

Next, when $\kappa\xi\in(0,1)$, we obtain
\beaa
|K_{\alpha}(\xi)| &\le& \dfrac{c^{\alpha}}{2 \Gamma(1-\alpha)} \left\{ e^{\kappa\xi} \int^{+\infty}_{0} \dfrac{e^{-\kappa s}}{s^{\alpha}}ds + e^{-\kappa \xi} \int^{1/\kappa}_{0} \dfrac{e^{\kappa s}}{s^{\alpha}} ds\right\} \\
&\le&  \dfrac{(c\kappa)^{\alpha}}{2\kappa \Gamma(1-\alpha)} \left\{ e\Gamma(1-\alpha) + \int^{1}_{0} \dfrac{e^{s}}{s^{\alpha}} ds\right\}.
\eeaa
Thus, the second integral part is evaluated as
\be\label{ine:mid}
\int^{1/\kappa}_{0} |K_{\alpha}(\xi)|d\xi \le  \dfrac{(c\kappa)^{\alpha}}{2\kappa^2 \Gamma(1-\alpha)} \left\{ e\Gamma(1-\alpha) + \int^{1}_{0} \dfrac{e^{s}}{s^{\alpha}} ds\right\}.
\ee

In the case that $\kappa\xi>1$, we deduce
\beaa
K_{\alpha}(\xi) &=& -\dfrac{c^{\alpha}}{2 \Gamma(1-\alpha)}\left\{ e^{\kappa\xi} \int^{+\infty}_{\xi} \dfrac{e^{-\kappa s}}{s^{\alpha}}ds -e^{-\kappa \xi} \int^{\xi}_{0} \dfrac{e^{\kappa s}}{s^{\alpha}} ds\right\} \\
&=& -\dfrac{(c\kappa)^{\alpha}}{2\kappa \Gamma(1-\alpha)}\left\{ e^{\kappa\xi} \int^{+\infty}_{\kappa\xi} \dfrac{e^{- s}}{s^{\alpha}}ds -e^{-\kappa \xi} \int^{\kappa\xi}_{0} \dfrac{e^{s}}{s^{\alpha}} ds\right\}.
\eeaa
We know that
\beaa
\int^{+\infty}_{\kappa\xi} \dfrac{e^{- s}}{s^{\alpha}}ds &=& \left[ -\dfrac{e^{-s}}{s^{\alpha}} \right]^{s=+\infty}_{{\red s=\kappa \xi}} - \alpha \int^{+\infty}_{\kappa\xi} \dfrac{e^{-s}}{s^{1+\alpha}} ds \\
&=& \dfrac{e^{-\kappa\xi}}{(\kappa\xi)^{\alpha}} - \alpha \int^{+\infty}_{\kappa\xi} \dfrac{e^{-s}}{s^{1+\alpha}} ds
\eeaa
and
\beaa
\int^{\kappa\xi}_{0} \dfrac{e^{s}}{s^{\alpha}} ds &=&  \int^{1}_{0} \dfrac{e^{s}}{s^{\alpha}} ds + \int^{\kappa\xi}_{1} \dfrac{e^{s}}{s^{\alpha}}ds \\
&=& \int^{1}_{0} \dfrac{e^{s}}{s^{\alpha}} ds +
\dfrac{e^{\kappa\xi}}{(\kappa\xi)^{\alpha}} - e  + \alpha \int^{\kappa\xi}_{1} \dfrac{e^{s}}{s^{1+\alpha}}ds.
\eeaa
Hence, we have
\beaa
K_{\alpha}(\xi) &=& \dfrac{(c\kappa)^{\alpha}}{2\kappa \Gamma(1-\alpha)}\left\{\alpha e^{\kappa\xi} \int^{+\infty}_{\kappa\xi} \dfrac{e^{- s}}{s^{1+\alpha}}ds +\alpha e^{-\kappa \xi} \int^{\kappa\xi}_{1} \dfrac{e^{s}}{s^{1+\alpha}} ds \right.\\
&& \hspace{6cm}\left. {\red +} \left( \int^{1}_{0}\dfrac{e^{s}}{s^{\alpha}}ds - e \right) e^{-\kappa\xi} \right\}.
\eeaa
Since we know that
\beaa
 \int^{+\infty}_{1/\kappa} e^{\kappa\xi} \int^{+\infty}_{\kappa\xi} \dfrac{e^{- s}}{s^{1+\alpha}}ds d\xi
 &=& -\dfrac{e}{\kappa}  \int^{+\infty}_{1} \dfrac{e^{- s}}{s^{1+\alpha}}ds + \int^{+\infty}_{1/\kappa} \dfrac{1}{(\kappa\xi)^{1+\alpha}}d\xi \\
 &=& -\dfrac{e}{\kappa}  \int^{+\infty}_{1} \dfrac{e^{- s}}{s^{1+\alpha}}ds + \dfrac{1}{\alpha \kappa} \\
 &\le & {\red \dfrac{1}{\alpha \kappa}} 
\eeaa
and
\beaa
    \int^{+\infty}_{1/\kappa} e^{-\kappa\xi} \int^{\kappa\xi}_{1} \dfrac{e^{s}}{s^{1+\alpha}}ds d\xi
    = \int^{+\infty}_{1/\kappa} \dfrac{d\xi}{(\kappa\xi)^{1+\alpha}} = \dfrac{1}{\alpha\kappa}
\eeaa
by using integration by parts,
we evaluate
\bea\label{ine:pos}
&& \int^{+\infty}_{1/\kappa} |K_{\alpha}(\xi)| d\xi \notag \\
&& \le \dfrac{(c\kappa)^{\alpha}}{2\kappa \Gamma(1-\alpha)} \left\{ \alpha \int^{+\infty}_{1/\kappa} e^{\kappa\xi} \int^{+\infty}_{\kappa\xi} \dfrac{e^{- s}}{s^{1+\alpha}}ds d\xi + \alpha \int^{+\infty}_{1/\kappa} e^{-\kappa\xi} \int^{\kappa\xi}_{1} \dfrac{e^{s}}{s^{1+\alpha}}dsd\xi
\right. \notag \\
&& \quad\quad \left. + \dfrac{e^{-1}}{\kappa} \left| \int^{1}_{0}\dfrac{e^{s}}{s^{\alpha}}ds - e \right| \right\} \notag \\
&&\le \dfrac{(c\kappa)^{\alpha}}{2\kappa^2 \Gamma(1-\alpha)} {\red \left\{ 2 + e^{-1} \left| \int^{1}_{0}\dfrac{e^{s}}{s^{\alpha}}ds - e \right|  \right\}}.
\eea

From inequalities \eqref{ine:neg}, \eqref{ine:mid}, and \eqref{ine:pos},
we can choose a sufficiently large $\Theta>0$ satisfying
\beaa
\|K_{\alpha}\|_{L^1} \le \Theta \kappa^{\alpha-2}.
\eeaa
\end{proof}

\medskip
Fix $\kappa>0$ sufficiently large satisfying $\|K_{\alpha}\|_{L^1}<1$.
Define {\red $\mathcal{K}_{\alpha}:L^{\infty}(\bR)\to L^{\infty}(\bR)$} as
\beaa
\mathcal{K}_{\alpha} \psi := (K_{\alpha}*\psi).
\eeaa
Also, we put $I$ as the identity operator in $L^{\infty}(\bR)$.
Then, $\mathcal{K}_{\alpha}$ is a bounded operator on {\red $L^{\infty}(\bR)$}, and the operator norm of $\mathcal{K}_{\alpha}$ is equal to $\|K_{\alpha}\|_{L^{1}}$.
Therefore, $(I-\mathcal{K}_{\alpha})$ is invertible on $L^{\infty}(\bR)$, and the inverse is represented by
\beaa
(I-\mathcal{K}_{\alpha})^{-1} = \dsum^{\infty}_{j=0} \mathcal{K}^{j}_{\alpha}
\eeaa
from the Neumann series theory.

Here, we apply the argument for \eqref{eq:integ} to obtain
\beaa
\psi = \left(\dsum^{\infty}_{j=0} (K_0* K^{[j]}_{\alpha} )\right)* g,
\eeaa
where $K^{[0]}_{\alpha}(\xi):=\delta(\xi)$ and $K^{[j]}_{\alpha}(\xi):=(K_{\alpha}*K^{[j-1]}_{\alpha})(\xi)$ for $j\ge 1$.

For simplicity, we denote
\be\label{green}
G(\xi) := \dsum^{\infty}_{j=0}(K_0* K^{[j]}_{\alpha})(\xi) = \left[ K_0*\left(\dsum^{\infty}_{j=0}K^{[j]}_{\alpha} \right) \right](\xi).
\ee
Then, $G(\xi)$ has the following property.

\begin{lemma}
    $G(\xi)$ belongs to $C(\bR)\cap L^1(\bR)$ and satisfies $\dint_{\bR} G(\xi)d\xi = \dfrac{1}{\kappa^2}$.
\end{lemma}
\begin{proof}
We note that $K_0, K_{\alpha}\in C(\bR)$.
Since the limit $G(\xi)= \dlim_{J\to\infty}\dsum^{J}_{j=0}(K_0*K^{[j]}_{\alpha})$ holds in the sense of uniform convergence, $G$ is continuous.
Moreover, we obtain
\beaa
\|G\|_{L^1} \le \|K_0\|_{L^1} \left( \dsum^{\infty}_{j=0} \|K_{\alpha}\|^{j}_{L^1} \right).
\eeaa
By using
\beaa
\int_{\bR}K_{\alpha}(\xi)d\xi = 0,
\eeaa
we have $\dint_{\bR}G(\xi)d\xi = \dint_{\bR}K_0(\xi)d\xi = \dfrac{1}{\kappa^2}$.
\end{proof}

\medskip
From the construction of $G$ and Lemma \ref{lem:integ}, we obtain the conclusion for the existence.
\begin{proposition}\label{lem:exi}
For $g\in C_{b}(\bR)$, $\psi(\xi)=(G*g)(\xi)$ belongs to $C^2_b(\bR)$ and is a solution to \eqref{eq:linear}.
\end{proposition}

\bigskip
In the next step, we establish the maximum principle for the operator $L$.
Let us consider the sign of the fractional derivative $\partial^{\alpha}_{\xi}$ at the minimal point.
\begin{lemma}\label{lem:deri}
Let $\psi\in C^1_b(\bR)$.
Assume that $\psi(\xi)$ attains the minimum at $\xi^{*}\in\bR$: $\psi(\xi^{*})=\min_{\xi\in\bR}{\red \psi(\xi)}$.
Then, we have
\beaa
\partial^{\alpha}_{\xi}\psi(\xi^{*}) = -\dfrac{\alpha}{\Gamma(1-\alpha)} \int^{+\infty}_{0} \dfrac{(\psi(\xi^{*}-s)-\psi(\xi^{*}))}{s^{1+\alpha}}ds\le 0.
\eeaa
Moreover, $\partial^{\alpha}_{\xi}\psi(\xi^{*})$ attains $0$ at $\xi^{*}\in\bR$ if and only if $\psi(\xi)=\psi(\xi^{*})$ for any $\xi\le\xi^{*}$.
\end{lemma}
\begin{proof}
From the fact that
\beaa
\dfrac{d}{ds} \left(\dfrac{\psi(\xi^{*}-s) - \psi(\xi^{*})}{s^{\alpha}} \right) = - \dfrac{\psi'(\xi^{*}-s)}{s^{\alpha}} - \dfrac{\alpha(\psi(\xi^{*}-s) - \psi(\xi^{*}))}{s^{1+\alpha}},
\eeaa
we obtain
\newpage
\beaa
\partial^{\alpha}_{\xi}\psi(\xi^{*}) &=& - \dfrac{1}{\Gamma(1-\alpha)}\left[ \dfrac{\psi(\xi^{*}-s) - \psi(\xi^{*})}{s^{\alpha}} \right]^{s\to +\infty}_{s\to +0} \\
&&\hspace{2cm}- \dfrac{\alpha}{\Gamma(1-\alpha)} \int^{+\infty}_{0} \dfrac{(\psi(\xi^{*}-s)-\psi(\xi^{*}))}{s^{1+\alpha}}ds \\
&=& -\dfrac{\alpha}{\Gamma(1-\alpha)} \int^{+\infty}_{0} \dfrac{(\psi(\xi^{*}-s)-\psi(\xi^{*}))}{s^{1+\alpha}}ds \le 0.
\eeaa
Furthermore, it is easy to see that
if $\psi(\xi)>\psi(\xi^{*})$ holds at the point $\xi\le \xi^{*}$, then $\partial^{\alpha}_{\xi}\psi(\xi^{*})$ must be negative from the continuity of $\psi(\xi)$.
Thus, $\psi(\xi)$ must be constant on $(-\infty,\xi^{*}]$ if $\partial^{\alpha}_{\xi} \psi(\xi^{*}) =0$.
\end{proof}

\medskip
Let us introduce the maximum principle.
\begin{proposition}\label{prop}
    Suppose that $\psi\in C^1_b(\bR)\cap C^2(\bR)$ satisfies
    \beaa
    L\psi(\xi) \le 0
    \eeaa
    and
    \beaa
    \liminf_{\xi\to-\infty}\psi(\xi)\ge 0,\quad 
    {\red \liminf_{\xi\to+\infty}\psi(\xi)\ge 0.}
    \eeaa
    Then, $\psi(\xi)$ is non-negative.
    {Furthermore, if $\psi(\xi)$ attains $0$ at the position $\xi^{*}\in\bR$, then $\psi(\xi)\equiv 0$ for all $\xi\le \xi^{*}$}.
\end{proposition}
\begin{proof}
    Suppose that $\psi(\xi)$ has a negative part.
    Then, from the assumption, there is $\xi^{*}\in\bR$ such that
    \beaa
    \psi(\xi^{*}) = \min_{\xi\in\bR}\psi(\xi) <0.
    \eeaa
    This implies
    \beaa
    L\psi(\xi^{*}) = \psi''(\xi^{*})- c^{\alpha}\partial^{\alpha}_{\xi} \psi(\xi^{*}) -\kappa^2 \psi(\xi^{*}) >0
    \eeaa
    from Lemma \ref{lem:deri}.
    This is a contradiction of the assumption $L\psi(\xi)\le 0$.

    Next, we assume that $\psi(\xi)$ is non-negative and attains 0 at $\xi^{*}\in\bR$.
    Then, we have the following.
    \beaa
    0\ge L\psi(\xi^{*}) = \psi''(\xi^{*}) - c^{\alpha}\partial^{\alpha}_{\xi}\psi(\xi^{*})\ge - c^{\alpha}\partial^{\alpha}_{\xi}\psi(\xi^{*}) \ge 0.
    \eeaa
    Hence, $\psi$ must be zero on $(-\infty,\xi^{*}]$ because $\partial^{\alpha}_{\xi}\psi(\xi^{*})=0$.
\end{proof}

\bigskip
Finally, we construct an upper solution and a lower solution to \eqref{eq:tw}.
We introduce a notion of upper and lower solutions.
\begin{definition}
For $c>0$, $\uphi\in C^2(\bR)$ is an upper solution to \eqref{eq:tw} if $\uphi$ satisfies
\beaa
c^{\alpha}\partial^{\alpha}_{\xi}\uphi(\xi) \ge \uphi''(\xi) + f(\uphi (\xi))
\eeaa
for all $\xi\in\bR$.
Conversely, $\lphi\in C^2(\bR)$ is a lower solution to \eqref{eq:tw} if $\lphi$ satisfies
\beaa
c^{\alpha}\partial^{\alpha}_{\xi}\lphi(\xi) \le \lphi''(\xi) + f(\lphi (\xi))
\eeaa
for all $\xi\in\bR$.
\end{definition}

\medskip
We focus on the property $u=0$ to construct the upper and lower solutions.
The linearized equation around $u=0$ is given as
\beaa
\dfrac{c^{\alpha}}{\Gamma(1-\alpha)} \int^{+\infty}_{0} \dfrac{\phi'(\xi-s)}{s^{\alpha}} ds = \phi''(\xi) + f'(0)\phi(\xi).
\eeaa
Since we have
\beaa
\dfrac{c^{\alpha}}{\Gamma(1-\alpha)} \int^{+\infty}_{0} \dfrac{e^{-\lambda s}}{s^{\alpha}} {\red ds} = \dfrac{(c\lambda)^{\alpha}}{\lambda},
\eeaa
the characteristic equation around $u=0$ is defined as follows:
\beaa
V(\lambda):= \lambda^2 -(c\lambda)^{\alpha}+f'(0)=0.
\eeaa
Then, we have

\begin{lemma}\label{lem:poly}
The following statements hold:
\begin{itemize}
    \item[(i)] For $c>c^{*}_\alpha$, there are two positive roots $\lambda_1<\lambda_2$ of $V(\lambda)=0$. Moreover, $V(\lambda)<0$ on $(\lambda_1,\lambda_2)$.
    \item[(ii)] When $c=c^{*}_{\alpha}$, there is a unique positive root $\lambda_1$.
    \item[(iii)] For $c\in (0,c^{*}_{\alpha})$, there are no positive roots.
\end{itemize}
\end{lemma}
\begin{proof}
It is obvious that $V(\lambda)$ is smooth on $(0,+\infty)$ and satisfies $V(0)=f'(0)>0$.
Since we obtain
\beaa
V'(\lambda) = 2\lambda - \dfrac{\alpha c^{\alpha}}{\lambda^{1-\alpha}},\quad
V''(\lambda) = 2  + \dfrac{\alpha (1-\alpha)c^{\alpha}}{\lambda^{2-\alpha}}
\eeaa
for $\lambda>0$,
$V(\lambda)$ is convex for $\lambda>0$ and attains the minimum at\\
$\lambda^{*}= (\alpha/2)^{1/(2-\alpha)} c^{\alpha/(2-\alpha)}$.
Moreover, we have
\beaa
V(\lambda^{*})= 2^{-2/(2-\alpha)} (\alpha-2)  (\alpha c^2)^{\alpha/(2-\alpha)}  + f'(0).
\eeaa
Thus, $V(\lambda)=0$ has two positive roots if $c>c^{*}_{\alpha}$, and no positive root when $c<c^{*}_{\alpha}$.
Moreover, $\lambda^{*}$ is a unique positive root of $V(\lambda)=0$ if $c=c^{*}_{\alpha}$.
\end{proof}

\medskip
\begin{remark}
    When we consider the linearized equation for Equation {\red \eqref{eq:main}} around $u=0$,
    \beaa
    \partial^{\alpha}_{t} w = w_{xx} + f'(0)w\quad (t>0,\ x\in\bR),
    \eeaa
    it has a solution $w(t,x)= E_{\alpha,1}((c\lambda t)^{\alpha}) e^{\lambda x}$ for all $c\ge c^{*}_{\alpha}$,
    where $\lambda$ is a positive root of $V(\lambda)=0$.
    From \eqref{mittag}, the asymptotic of $w(t,x)$ is approximated by
    \beaa
    w(t,x) \simeq \dfrac{1}{\alpha} e^{\lambda (x+ct)} =  \dfrac{1}{\alpha} e^{\lambda \xi} \quad (t\to +\infty).
    \eeaa
    This consideration suggests that the asymptotic traveling wave solution has an exponential tail at $u=0$.
\end{remark}

\medskip
Next, let us construct an upper solution.
For convenience, we set
\beaa
N_1(\xi) &:=& \uphi''(\xi) - c^{\alpha}\partial^{\alpha}_{\xi}\uphi(\xi) + f(\uphi (\xi)), \\
N_2(\xi) &:=& \lphi''(\xi) - c^{\alpha}\partial^{\alpha}_{\xi}\lphi(\xi) + f(\lphi (\xi)).
\eeaa
Let $\rho$ be a positive mollifier, that is, $\rho(x)$ is an even function satisfying
\beaa
\rho\in C^{\infty}_{0}(\bR),\quad \rho(x)
\begin{cases}
>0 & |x|<1, \\
=0 & |x|\ge 1,
\end{cases}
\quad \int_{\bR}\rho(x)dx =1.
\eeaa
We set $\rho_{\ep}(x) := \dfrac{1}{\ep}\rho\left(\dfrac{x}{\ep}\right)$.
For $c> c^{*}_{\alpha}$, we define
\beaa
\upsi(\xi) =
\begin{cases}
e^{\lambda_{1} \xi} &\xi<0, \\
1 &{\red \xi\ge 0},
\end{cases}
\eeaa
where $\lambda_1$ is a minimal positive root of $V(\lambda)$.
Then, we set $\uphi(\xi)=\uphi(\xi;\ep):=(\rho_{\ep}*\upsi)(\xi)$.
It is easy to see that we have
\beaa
\uphi(\xi) =
\begin{cases}
R(\ep\lambda_1) e^{\lambda_1 \xi} &\xi\le -\ep \\
\int^{\xi/\ep}_{-1}\rho(y)dy + e^{\lambda_1\xi} \int^{1}_{\xi/\ep} \rho(y) e^{-\ep \lambda_1 y}dy &|\xi|<\ep \\
1 &\xi\ge\ep,
\end{cases}
\eeaa
where $R(\lambda):={\red \int^{1}_{-1} \rho(y) e^{-\lambda y}dy}$.
\begin{lemma}\label{lem:up}
For $c> c^{*}_{\alpha}$ and $\ep>0$, $\uphi$ defined as above satisfies the following properties:
\begin{itemize}
    \item[(i)] $\uphi\in C^\infty(\bR)$;
    \item[(ii)] $\uphi$ is non-decreasing;
    \item[(iii)] $\uphi(\xi)\in [0,1]$ for all $\xi\in\bR$.
\end{itemize}
Moreover, $\uphi$ is an upper solution to \eqref{eq:tw} if $\ep$ is sufficiently small.
\end{lemma}
\begin{proof}
(i) is derived from the general theory of the mollifier.
Since $\upsi(\xi)$ is a non-decreasing function connecting $0$ to $1$, we have (ii) and (iii).

In the rest of the proof, we show that $\uphi$ is an upper solution to \eqref{eq:tw}.
For $\xi\le -\ep$, we have
\beaa
N_1(\xi) &=& R(\ep\lambda_1) \left( \lambda^2_1 e^{\lambda_1 \xi} - \dfrac{c^{\alpha}\lambda_1}{\Gamma(1-\alpha)} \int^{+\infty}_{0} \dfrac{e^{\lambda_1(\xi-s)}}{s^{\alpha}}ds \right) + f(\uphi(\xi))  \\
&=& R(\ep\lambda_1)e^{\lambda_1\xi} \left\{ \lambda^2_1  - (c\lambda_1)^{\alpha} \right\} + f(\uphi(\xi)) \\
&=&  f(\uphi(\xi)) - f'(0) \uphi(\xi) \le 0
\eeaa
from the fact that $V(\lambda_1)=0$ and assumption \eqref{ass:non}.

When $\xi \ge \ep$, we obtain
\beaa
N_1(\xi) = - \dfrac{c^{\alpha}}{\Gamma(1-\alpha)} \int^{+\infty}_{0} \dfrac{\uphi'(\xi-s)}{s^{\alpha}} ds < 0
\eeaa
because $\uphi$ is non-decreasing.
Thus, for any $\ep>0$, we have $N_1(\xi)\le 0$ on $|\xi|\ge\ep$.

In the case that $|\xi|<\ep$, we put $\eta=\xi/\ep$.
Then, the first and second derivatives satisfy
\beaa
\uphi'(\xi) &=& \lambda_1 e^{\lambda_1\xi} \int^{1}_{\xi/\ep} \rho({\red y})e^{-\ep\lambda_1 y} dy \\
&=&  \lambda_1 e^{\ep\lambda_1 \eta} \int^{1}_{\eta} \rho(y)e^{-\ep\lambda_1 y} dy = \lambda_1 \int^{1}_{\eta} \rho(y)dy + O(\ep), \\
\uphi''(\xi) &=& \lambda^2_1 e^{\lambda_1\xi} \int^{1}_{\xi/\ep} \rho(y)e^{-\ep\lambda_1 y} dy - \dfrac{\lambda_1}{\ep}\rho\left(\dfrac{\xi}{\ep}\right) \\
&=& \lambda^2_1 e^{\ep\lambda_1\eta} \int^{1}_{\eta} \rho(y)e^{-\ep\lambda_1 y} dy - \dfrac{\lambda_1}{\ep}\rho\left(\eta\right) = - \dfrac{\lambda_1}{\ep}\rho(\eta) + \lambda^2_1 \int^{1}_{\eta} \rho(y)dy + O(\ep)
\eeaa
for $\eta\in[-1,1]$ if $\ep$ is sufficiently small.
Furthermore, for any $\eta\in[-1,1]$ and sufficiently small $\ep>0$, we approximate
\beaa
f(\uphi(\xi)) = f(1) + O(\ep) = O(\ep)
\eeaa
and
\beaa
\partial^{\alpha}_{\xi}\uphi(\xi) &=& \dfrac{1}{\Gamma(1-\alpha)} \left\{ \int^{+\infty}_{\ep(\eta+1)} \dfrac{\uphi'(\ep\eta-s)}{s^{\alpha}} ds + \int^{\ep(\eta+1)}_{0} \dfrac{\uphi'(\ep\eta-s)}{s^{\alpha}} ds \right\} \\
&=& \dfrac{\lambda_1 R(\ep\lambda_1)}{\Gamma(1-\alpha)} \int^{+\infty}_{\ep(\eta+1)} \dfrac{{\red e^{\lambda_1(\ep\eta-s)}}}{s^{\alpha}} ds  + {\red O(\ep^{1-\alpha})} = \lambda^{\alpha}_1 + {\red O(\ep^{1-\alpha})}.
\eeaa
Thus, we deduce
\beaa
N_1(\xi) &=& \uphi''(\xi) - c^{\alpha}\partial^{\alpha}_{\xi}\uphi(\xi) + f(\uphi (\xi)) \\
&=& - \dfrac{\lambda_1}{\ep}\rho(\eta) + \lambda^2_1 \int^{1}_{\eta} \rho(y)dy - (c\lambda_1)^{\alpha}  +{\red O(\ep^{1-\alpha})}
\eeaa
for all $\eta\in (-1,1)$ and sufficiently small $\ep>0$.
Since we know $\rho(\eta)\ge 0$ and have
\beaa
\lambda^2_1 \int^{1}_{\eta} \rho(y)dy - (c\lambda_1)^{\alpha} \le \lambda^2_1 - (c\lambda_1)^{\alpha} = -f'(0) < 0,
\eeaa
$N_1(\xi)$ is negative on $(-\ep,\ep)$ as long as $\ep$ is sufficiently small.
\end{proof}

\medskip
Next, we construct a lower solution.
For $c>c^{*}_{\alpha}$, let $\psi(\xi)= e^{\lambda_1\xi}-he^{\nu\lambda_1\xi}$ with $\nu\in (1,\min\{1+a,\lambda_2/\lambda_1\})$,
where $h>1$ is a sufficiently large constant satisfying
\be\label{condi:h}
{\red h^{a/(\nu-1)} \ge -\dfrac{M \{R(\ep\lambda_1)\}^{1+a}}{R(\ep\nu\lambda_1)V(\nu\lambda_1)}.}
\ee
We remark that the right-hand side is strictly positive for any $\ep\in[0,1)$.
We fix $h$, satisfying \eqref{condi:h} for all $\ep\in[0,1)$.
When we set $\xi_{0}:= -\dfrac{\log h}{\lambda_1(\nu-1)}<0$, we have
\beaa
\psi(\xi)
\begin{cases}
>0 & \xi<\xi_0 \\
=0 & \xi=\xi_0 \\
<0 & \xi>\xi_0.
\end{cases}
\eeaa
Now, we set
\beaa
\lpsi(\xi) =
\begin{cases}
e^{\lambda_{1} \xi} - h e^{ \nu \lambda_1 \xi} &\xi< \xi_0, \\
0 &{\red \xi\ge\xi_0}
\end{cases}
\eeaa
and define $\lphi(\xi)= \lphi(\xi;\ep) := (\rho_{\ep}*\lpsi)(\xi)$.
From direct computation, we obtain
\beaa
\lphi(\xi) =
\begin{cases}
R(\ep\lambda_1)e^{\lambda_1\xi}  -h R(\ep\nu\lambda_1) e^{\nu \lambda_1\xi} &\xi \le \xi_0 - \ep, \\
e^{\lambda_1\xi} \int^{1}_{(\xi-\xi_0)/\ep} \rho(y)e^{-\ep\lambda_1 y}dy - h e^{\nu\lambda_1\xi} \int^{1}_{(\xi-\xi_0)/\ep} \rho(y)e^{-\ep\nu\lambda_1 y}dy &|\xi-\xi_0|<\ep \\
0 & \xi\ge \xi_0 + \ep.
\end{cases}
\eeaa

\begin{lemma}\label{lem:low}
For $c> c^{*}_{\alpha}$ and $\ep>0$, $\lphi$ defined as above satisfies the following properties:
\begin{itemize}
    \item[(i)] $\lphi\in C^\infty(\bR)$;
    \item[(ii)] $\lphi$ is non-negative on $\bR$ and positive on $(-\infty,\xi_0+\ep)$;
    \item[(iii)] $\lphi(\xi) \le \uphi(\xi)$ for all $\xi\in\bR$.
\end{itemize}
Moreover, $\lphi$ is a lower solution to \eqref{eq:tw} if $\ep$ is sufficiently small.
\end{lemma}
\begin{proof}
(i) holds from the general theory of the mollifier.
$\lphi$ is non-negative because $\lpsi$ is non-negative.
Also, $\lpsi$ is positive on {\red $(-\infty,\xi_0)$}, and thus $\lphi$ is positive on $(-\infty,\xi_0+\ep)$.
Hence, we have (ii).
(iii) is derived from the fact that $\lpsi< \upsi$ on $\bR$.


\medskip
Next, when $\xi<\xi_0-\ep$,
we deduce
{\small \beaa
N_2(\xi) &=& \lambda^2_1 R(\ep\lambda_1) e^{\lambda_1\xi} - h(\nu\lambda_1)^2 R(\ep\nu \lambda_1) e^{\nu\lambda_1\xi} + f(R(\ep\lambda_1) e^{\lambda_{1} \xi} -  h R(\ep\nu\lambda_1) e^{ \nu \lambda_1 \xi}) \\
&&  - \dfrac{c^{\alpha}}{\Gamma(1-\alpha)} \left\{ \lambda_1 R(\ep\lambda_1)\int^{+\infty}_{0} \dfrac{e^{\lambda_1(\xi-s)}}{s^{\alpha}}ds - h\nu\lambda_1 R(\ep\nu\lambda_1)\int^{+\infty}_{0} \dfrac{e^{\nu\lambda_1(\xi-s)}}{s^{\alpha}}ds  \right\} \\
&=& {\red  \lambda^2_1 R(\ep\lambda_1) e^{\lambda_1\xi} - h(\nu\lambda_1)^2 R(\ep\nu\lambda_1) e^{\nu\lambda_1\xi} +f(R(\ep\lambda_1) e^{\lambda_{1} \xi} -  h R(\ep\nu\lambda_1) e^{ \nu \lambda_1 \xi})} \\
&& {\red \quad\quad- (c\lambda_1)^{\alpha}R(\ep\lambda_1)e^{\lambda_1\xi} + h(c\nu\lambda_1)^{\alpha} R(\ep\nu\lambda_1) e^{\nu\lambda_1\xi} }\\
&=& f(R(\ep\lambda_1) e^{\lambda_{1} \xi} -  h R(\ep\nu\lambda_1) e^{ \nu \lambda_1 \xi}) - f'(0)(R(\ep\lambda_1) e^{\lambda_{1} \xi} -  h R(\ep\nu\lambda_1) e^{ \nu \lambda_1 \xi}) \\
&& \quad \quad - R(\ep\nu\lambda_1) hV(\nu\lambda_1)e^{\nu\lambda_1\xi} \\
&\ge & {- M (R(\ep\lambda_1) e^{\lambda_{1} \xi} -  h R(\ep\nu\lambda_1) e^{ \nu \lambda_1 \xi})^{1+a}} - R(\ep\nu\lambda_1) hV(\nu\lambda_1)e^{\nu\lambda_1\xi}  \\
&\ge & - M \{R(\ep\lambda_1)\}^{1+a}e^{(1+a)\lambda_1 \xi} - hR(\ep\nu\lambda_1)V(\nu\lambda_1)e^{\nu\lambda_1\xi}.
\eeaa}We note that $V(\nu\lambda_1)<0$.
Since it is easy to show that the last function is {\red nonnegative} if
\beaa
\xi \le \dfrac{1}{(1+a-\nu)\lambda_1} \log\left( \dfrac{-h R(\ep\nu\lambda_1)V(\nu\lambda_1)}{M \{R(\ep\lambda_1)\}^{1+a} } \right) =: \xi^{*},
\eeaa
we obtain $N_2(\xi)\ge 0$ for all $\xi<{\red \xi_0-\ep}$ when $\xi_0\le \xi^{*}$.
It is easy to show that $\xi_0\le \xi^{*}$ holds if and only if $h$ satisfies \eqref{condi:h}.
Then, we obtain $N_2(\xi)\ge 0$ for $\xi<{\red \xi_0-\ep}$.

Next, we consider the case that $\xi\ge \xi_0 +\ep$.
Since $\lphi(\xi)\equiv 0$ on $[\xi_0+\ep,+\infty)$ and $\min_{\xi\in\bR}\lphi(\xi)=0$,
we have
\beaa
N_2(\xi) = -c^{\alpha}\partial^{\alpha}_{\xi}\lphi(\xi) > 0
\eeaa
by using Lemma \ref{lem:deri}.

Finally, in the case that $|\xi-\xi_0|<\ep$, we apply the same argument of the proof of Lemma \ref{lem:up} to show.
Let $\eta = (\xi-\xi_0)/\ep$.
Then, for sufficiently small $\ep>0$, the first derivative is approximated by
\beaa
\lphi'(\xi) &=& \lambda_1 e^{\lambda_1\xi} \int^{1}_{(\xi-\xi_0)/\ep} \rho(y)e^{-\ep\lambda_1 y}dy - h\nu\lambda_1 e^{\nu\lambda_1\xi} \int^{1}_{(\xi-\xi_0)/\ep} \rho(y)e^{-\ep\nu\lambda_1 y}dy \\
&=& \lambda_1 e^{\lambda_1(\xi_0+\ep\eta) } \int^{1}_{\eta} \rho(y)e^{-\ep\lambda_1 y}dy - h\nu\lambda_1 e^{\nu\lambda_1(\xi_0+\ep\eta)} \int^{1}_{\eta} \rho(y)e^{-\ep\nu\lambda_1 y}dy \\
&=& \lambda_1 e^{\lambda_1 \xi_0} (1-\nu)\int^{1}_{\eta}\rho(y)dy  + O(\ep).
\eeaa
Here, we use the fact that $e^{\lambda_1\xi_0}=he^{\nu\lambda_1\xi_0}$.
Also, the approximation of the second derivative is computed as
\beaa
\lphi''(\xi) &=& \lambda^2_1 e^{\lambda_1\xi} \int^{1}_{(\xi-\xi_0)/\ep} \rho(y)e^{-\ep\lambda_1 y}dy - h(\nu\lambda_1)^2 e^{\nu\lambda_1\xi} \int^{1}_{(\xi-\xi_0)/\ep} \rho(y)e^{-\ep\nu\lambda_1 y}dy \\
&&\quad\quad + \dfrac{\lambda_1 e^{\lambda_1\xi_0} {\red (\nu-1)}}{\ep} \rho\left(\dfrac{\xi-\xi_0}{\ep}\right) \\
&=& \dfrac{\lambda_1 e^{\lambda_1\xi_0} (\nu-1)}{\ep} \rho(\eta)
+ e^{\lambda_1\xi_0}\lambda^2_1 ( 1 -\nu^2) \int^{1}_{\eta}\rho(y)dy + O(\ep).
\eeaa
Moreover, we obtain $f(\lphi(\xi))= O(\ep)$ and
\beaa
\partial^{\alpha}_{\xi}\lphi(\xi) &=& \dfrac{1}{\Gamma(1-\alpha)} \left[\int^{+\infty}_{\xi-\xi_0+\ep} + \int^{\xi-\xi_0+\ep}_0 \right] \dfrac{\lphi'(\xi-s)}{s^{\alpha}}ds \\[2mm]
&=& \dfrac{1}{\Gamma(1-\alpha)} \left[\int^{+\infty}_{\ep(\eta+1)} + \int^{\ep(\eta+1)}_0 \right] \dfrac{\lphi'(\ep\eta+\xi_0-s)}{s^{\alpha}}ds \\[2mm]
&=& \dfrac{1}{\Gamma(1-\alpha)} \int^{+\infty}_{\ep(\eta+1)} \dfrac{\lphi'(\ep\eta+\xi_0-s)}{s^{\alpha}}ds + {\red O(\ep^{1-\alpha})} \\[2mm]
&=& \dfrac{\lambda_1 e^{\lambda_1\xi_0}}{\Gamma(1-\alpha)}  \left\{R(\ep\lambda_1) \int^{+\infty}_{\ep(\eta+1)} \dfrac{e^{\lambda_1(\ep\eta -s)}}{s^{\alpha}}ds -
\nu R(\ep\nu\lambda_1) \int^{+\infty}_{\ep(\eta+1)} \dfrac{e^{\nu\lambda_1(\ep\eta -s)}}{s^{\alpha}}ds
\right\} + {\red O(\ep^{1-\alpha})} \\[2mm]
&=& \lambda^{\alpha}_1 e^{\lambda_1\xi_0}(1 - \nu^{\alpha}) + {\red O(\ep^{1-\alpha})}
\eeaa
for all $\eta\in(-1,1)$ and sufficiently small $\ep>0$.
Thus, we have
\beaa
N_2(\xi) &{\red =}& \dfrac{\lambda_1 e^{\lambda_1\xi_0} (\nu-1)}{\ep} \rho(\eta)
- {\red e^{\lambda_1\xi_0} \left\{ ( \nu^2-1)\lambda^2_1 \int^{1}_{\eta}\rho(y)dy - ( \nu^\alpha-1) (c\lambda_1)^{\alpha} \right\} + O(\ep^{1-\alpha})}.
\eeaa
Since the $O(1)$ term satisfies
\beaa
{\red (\nu^2-1)\lambda^2_1 \int^{1}_{\eta}\rho(y)dy - ( \nu^\alpha-1) (c\lambda_1)^{\alpha} \le (\nu^2-1)\lambda^2_1  - ( \nu^\alpha-1) (c\lambda_1)^{\alpha}
= V(\nu\lambda_1)} < 0,
\eeaa
$N_2(\xi)$ is positive on $|\xi-\xi_0|<\ep$ as long as $\ep$ is sufficiently small.
\end{proof}

\bigskip
\subsection{Proof of Theorem \ref{thm:main} in the case $c>c^{*}_{\alpha}$}\label{subsec:fast}

Fix $c>c^{*}_{\alpha}$.
We give a sufficiently small $\ep>0$ such that $\uphi$ and $\lphi$ become upper and lower solutions to \eqref{eq:tw}, respectively.
Let us define the sequence of functions
\beaa
\begin{cases}
\psi_0 = \uphi, \\
\psi_j = G*(\kappa^2 \psi_{j-1} + f(\psi_{j-1})) &(j=1,2,\ldots),
\end{cases}
\eeaa
where $G$ is defined by \eqref{green}.

\begin{lemma}\label{lem:seq}
For $j\in\bN\cup\{0\}$, the following hold:
\begin{itemize}
    \item[(i)] $\psi_j \in C^2_b(\bR)$;
    \item[(ii)] $\dlim_{\xi\to -\infty} \psi_j(\xi)=0$ and $\dlim_{\xi\to+\infty}\psi_j(\xi)=1$ hold;
    \item[(iii)] $\lphi(\xi) \le \psi_{j}(\xi) \le \uphi(\xi)\ (\xi\in\bR)$;
    \item[(iv)] $\psi_j$ is non-decreasing function;
\end{itemize}
Furthermore, $\psi_j$ is an increasing function for $j\in\bN$.
\end{lemma}
\begin{proof}
Since it is obvious that (i)-(iv) hold in the case $j=0$,
we prove them by induction.
Suppose that the statements (i)-(iv) hold for $j=k\ge 0$.

(i) From the definition of $\psi_{k+1}$ and Proposition \ref{lem:exi},
$\psi_{k+1}\in C^2_b(\bR)$ holds.

(ii) Since we know that $\psi_k \in C_b$ and $\dlim_{\xi\to+\infty} (\kappa^2 \psi_k(\xi) + f(\psi_k(\xi))) = \kappa^2$,
we deduce
\beaa
\dlim_{\xi\to+\infty} \psi_{k+1} (\xi) = \dlim_{\xi\to+\infty} G*(\kappa^2 \psi_{k} + f(\psi_{k}))(\xi) = \kappa^2 \int_{\bR}G(\xi)d\xi = 1
\eeaa
from the dominated convergence theorem.
The same argument also provides the result of the case $\xi\to-\infty$.

(iii) Since $\psi_{k+1}$ is a solution to
\beaa
L \psi_{k+1} + \kappa^2 \psi_{k} + f(\psi_k) =0,
\eeaa
we have
\beaa
L\lphi \ge L \psi_{k+1} \ge L \uphi
\eeaa
from the monotonicity of $\kappa^2 u + f(u)$.
Applying Proposition \ref{prop}, statement (iii) holds.

(iv) Let $\psi_{k+1,p}(\xi):= \psi_{k+1}(\xi-p)$ for $p>0$.
Then, we deduce
\beaa
L \psi_{k+1,p}(\xi) &=& -\kappa^2 \psi_{k}(\xi-p) - f(\psi_{k}(\xi-p)) \\
&\ge& -\kappa^2 \psi_{k}(\xi) - f(\psi_{k}(\xi)) \\
&=& L \psi_{k+1}(\xi)
\eeaa
for $\xi\in\bR$.
From  Proposition \ref{prop}, we obtain $\psi_{k+1,p}(\xi)\le \psi_{k+1}(\xi)$ for all $\xi\in\bR$.
This implies that $\psi_{k+1}$ is non-decreasing.

Furthermore, if there is $\xi^{*}\in\bR$ such that $\psi_{k+1,p}(\xi^{*})= \psi_{k+1}(\xi^{*})$ holds,
then by applying Proposition \ref{prop}, we deduce
\beaa
\psi_{k+1}(\xi)=\psi_{k+1}(\xi-p)
\eeaa
for all $\xi\le \xi^{*}$.
Since we have Lemma \ref{lem:seq} (ii) and $\dlim_{\xi\to-\infty} \uphi(\xi)=\dlim_{\xi\to-\infty}\lphi(\xi)=0$,
$\psi_{k+1}$ must be zero on $(-\infty,\xi^{*}]$.
However, this is a contradiction of Lemma \ref{lem:seq} {\red (iii)} because $\lphi$ is non-zero on {\red $(-\infty,\xi_0+\ep)$}.
Thus, $\psi_{k+1}(\xi)>\psi_{k+1}(\xi-p)$ holds for all $p>0$ and $\xi\in\bR$.
This implies that $\psi_{k+1}$ is an increasing function.
\end{proof}

\medskip
\begin{lemma}\label{lem:ord}
    For $j\in\bN$, we have
    $\psi_{j}(\xi) \le \psi_{j-1}(\xi)$ for all $\xi\in\bR$.
\end{lemma}
\begin{proof}
The statement holds for the case $j=1$ from Lemma \ref{lem:seq} (iii).
Suppose that the statement holds for $j=k\ge 1$.
From the definition of $\psi_{j}\ (j\in\bN)$ and Proposition \ref{lem:exi},
we deduce
\beaa
L\psi_{k+1} = -\kappa^2 \psi_{k} - f(\psi_{k})
\ge  -\kappa^2 \psi_{k-1} - f(\psi_{k-1}) = L \psi_{k}.
\eeaa
By applying Proposition \ref{prop}, we obtain the statement.
\end{proof}

\medskip
Since $\{\psi_j\}_{j\in\bN}$ is a monotone sequence from Lemmas \ref{lem:seq} and \ref{lem:ord}, it converges point-wise to {\red a non-decreasing function} $\phi$ satisfying
\be\label{ine:conv}
\lphi(\xi) \le \phi(\xi) \le \uphi(\xi)
\ee
from the monotone convergence theorem and Helly's selection theorem.
Furthermore, we have
\beaa
\phi(\xi) = \dlim_{j\to\infty} \psi_j(\xi) = \dlim_{j\to\infty} G*(\kappa^2 \psi_{j-1}+f(\psi_{j-1}))(\xi) = G*(\kappa^2 \phi+f(\phi))(\xi)
\eeaa
for all $\xi\in\bR$ by using the dominated convergence theorem.
Since $G$ is continuous and $\phi$ is bounded,
$\phi\in C_b(\bR)$ holds.
Hence, $\phi$ belongs to $C^2_b(\bR)$ and is a solution to \eqref{eq:tw} from Proposition \ref{lem:exi}.

{\red It is easy to see that $\phi$ is an increasing function from the same argument of the proof of Lemma \ref{lem:seq}.}
From \eqref{ine:conv} and Lemmas \ref{lem:up} and \ref{lem:low},
we obtain $\dlim_{\xi\to-\infty}\phi(\xi)=0$.
Moreover, $\phi$ is bounded and increasing,
thus we know that there exists a limit $\phi^{+}:= \dlim_{\xi\to+\infty}\phi(\xi)$ satisfying $0\le \phi^{+}\le 1$.
Then, the condition of $\phi^{+}$ is derived as
\beaa
\phi^{+}=\dlim_{\xi\to+\infty}\phi(\xi) = \dlim_{\xi\to+\infty} G*(\kappa^2 \phi+f(\phi))(\xi) = \dfrac{1}{\kappa^2} (\kappa^2 \phi^{+}+f(\phi^{+})).
\eeaa
This implies $f(\phi^{+})=0$, and so $\phi^{+}=1$ holds.
Therefore, $\phi$ is a solution to \eqref{eq:tw} satisfying \eqref{boundary}.

Finally, from Lemmas \ref{lem:up} and \ref{lem:low},
we have
\beaa
\dlim_{\xi\to-\infty} e^{-\lambda_1\xi} \phi(\xi)= R(\ep\lambda_1).
\eeaa
By translating $\phi$ appropriately, the limit can be determined to be $1$.
Thus, Theorem \ref{thm:main} for $c>c^{*}_{\alpha}$ is proved.

\medskip
\subsection{Proof of Theorem \ref{thm:main} in the case $c=c^{*}_{\alpha}$}
Let $\{c_j\}_{j\in\bN}$ be a minimizing sequence for $c^{*}_{\alpha}$ and $\{\phi_j\}_{j\in\bN}$ be the solution to \eqref{eq:tw} corresponding to $c_j>c^{*}_{\alpha}$ constructed by Subsection \ref{subsec:fast}.
For convenience, we take $\phi_j(0)= \frac{1}{2}$ for all $j\in\bN$ by shifting the spatial variable.
Since $\{\phi_j\}_{j\in\bN}$ is a bounded sequence of increasing functions in $C^{2,\beta}(\bR)$ from the similar argument for the elliptic regularity,
there exists a subsequence $\{\phi_{j_k} \}_{k\in\bN}$ such that $\phi_{j_k}$ converges to {\red a non-decreasing} function $\phi^{*}$ in {\red $C^{2}_{loc}(\bR)$} which satisfies \eqref{eq:tw} with $c=c^{*}_{\alpha}$.
Furthermore, $\phi^{*}$ have the limits $\phi^{*\pm} = \dlim_{\xi\to\pm\infty}\phi^{*}(\xi)$, and so $\phi^{*\pm}$ satisfy $f(\phi^{*\pm})=0$ and $0\le \phi^{*-}< \phi^{*+}\le 1$.
This implies $\phi^{*-}=0$ and $\phi^{*+}=1$.
{\red Moreover, $\phi^*$ is an increasing function from the same argument of the proof of Lemma \ref{lem:seq}.}
Thus, $\phi^{*}$ is a desired solution.

\bigskip
\section{Numerical experiments}

Numerical experiments are performed to investigate the spatio-temporal dynamics of the solution to equation \eqref{eq:main}.

\subsection{Numerical scheme and setting}\label{subsec:num}
We first introduce the numerical scheme for the simulation.

Let $\delta_t>0$ and $u_j(x):= u(j\delta_t,x)$ for $j\ge 0$.
We apply the time difference proposed by \cite{LX}:
\beaa
\partial^{\alpha}_t u((j+1)\delta_t,x) &=& \dfrac{1}{\Gamma(1-\alpha)} \dsum^{j}_{m=0} \int^{(m+1)\delta_t}_{m\delta_t} \dfrac{\partial u}{\partial s}(s,x) \dfrac{ds}{\{((j+1)\delta_t- s)\}^{\alpha}} \\
&\simeq& \dfrac{1}{\Gamma(1-\alpha)} \dsum^{j}_{m=0}
\dfrac{u_{m+1}(x)-u_m(x)}{\delta_t} \int^{(m+1)\delta_t}_{m\delta_t} \dfrac{ds}{\{((j+1)\delta_t- s)\}^{\alpha}}.
\eeaa
Following the straightforward computation of \cite{LX},
the discretization is represented by
\beaa
\partial^{\alpha}_t u((j+1)\delta_t,x)
&\simeq& {\red \dfrac{1}{\Gamma(2-\alpha)\delta^{\alpha}_t } \dsum^{j}_{m=0}(u_{j+1-m}(x)-u_{j-m}(x))} [(m+1)^{1-\alpha}-m^{1-\alpha}].
\eeaa
Thus, when $u_{m}(x)\ (m=0,1,\ldots,j)$ are given,
we obtain $u_{j+1}(x)$ by using the time difference equation
\beaa
{\red \dfrac{1}{\Gamma(2-\alpha) \delta^{\alpha}_t} \dsum^{j}_{m=0}
(u_{j+1-m}(x)-u_{j-m}(x))} [(m+1)^{1-\alpha}-m^{1-\alpha}] = (u_{j+1})_{xx}(x) + f(u_{j}(x)).
\eeaa

Next, let $\delta_x>0$, $k_{x}\in\bN$, and $u_{j,k}:= u(j\delta_t,k\delta_{x})$ for $0 \le k \le k_x$.
For $0\le k \le k_x$,
we approximate $(u_{j+1})_{xx}(k\delta_x)$ as
\beaa
(u_{j+1})_{xx}(k\delta_x) \simeq \dfrac{u_{j+1,k+1}-2 u_{j+1,k} + u_{j+1,k-1}}{\delta^2_x}.
\eeaa
Thus, spatial discretization yields
\beaa
&& {\red \dfrac{1}{\Gamma(2-\alpha)\delta^{\alpha}_t} \dsum^{j}_{m=0}
(u_{j+1-m,k}-u_{j-m,k})}[(m+1)^{1-\alpha}-m^{1-\alpha}] \\
&&\hspace{3cm} = \dfrac{u_{j+1,k+1}-2 u_{j+1,k} + u_{j+1,k-1}}{\delta^2_x} + f(u_{j,k}).
\eeaa
When considering the Neumann boundary condition
\beaa
\dfrac{\partial u}{\partial x}(0) = \dfrac{\partial u}{\partial x}(k_x\delta_x)=0,
\eeaa
we set $u_{j,-1}:=u_{j,1}$ and $u_{j,k_x+1}:=u_{j,k_x-1}$ for $j\ge 0$.
Numerical experiments are performed with this scheme.

For the simulation, we solve the problem
\beaa
\begin{cases}
\partial^{\alpha}_{t} u = u_{xx} + u (1-u)\quad (t>0,\ 0<x<l), \\
\dfrac{\partial u}{\partial x}(t,0) = \dfrac{\partial u}{\partial x}(t,l)=0\quad (t>0), \\
u(0,x)=u_0(x)\quad (0<x<l),
\end{cases}
\eeaa
where $l>0$ and
\beaa
u_0(x)=
\begin{cases}
0 & 0<x<l_0, \\
\omega & l_0\le x < l
\end{cases}
\eeaa
for some $l_0\in(0,l)$.
We compute the profile of the solution when $T>0$ in the following step.
\begin{itemize}
    \item[(i)] Compute the point $x^{*}(t)$ such that $u(t,x^{*}(t))=0.1$ {\red if it exists for $t>0$};
    \item[(ii)] Find the minimal $T>0$ such that $x^{*}(T)<x_0$ for a given constant $x_0\in (0,l)$.
\end{itemize}
Moreover, the propagation speed of a numerical solution for $\alpha$ is computed by
\beaa
c^{num}_{\alpha} :={\red \left| \dfrac{x^{*}(T)-x^{*}(T-\delta_t)}{\delta_t} \right|}.
\eeaa

\medskip
\subsection{Numerical results}
The profile of the numerical solution is shown in Fig. \ref{fig:pro}.
Initially, the front of the solution appears and moves from right to left in the spatio-temporal plots.
In addition, the position of the front is linear with respect to time, indicating that the propagation advances at a rate $O(t)$.
This suggests that the solutions exhibit behavior similar to that of a traveling wave solution as time progress.
Moreover, we see that the front profile decays to 0 rapidly and approaches 1 slowly.
This expected behavior reflects kinetics, where the system leaves the unstable state exponentially and approaches the stable state with an algebraic decay.

\begin{figure}[bt]
\begin{center}
	\includegraphics[width=12cm]{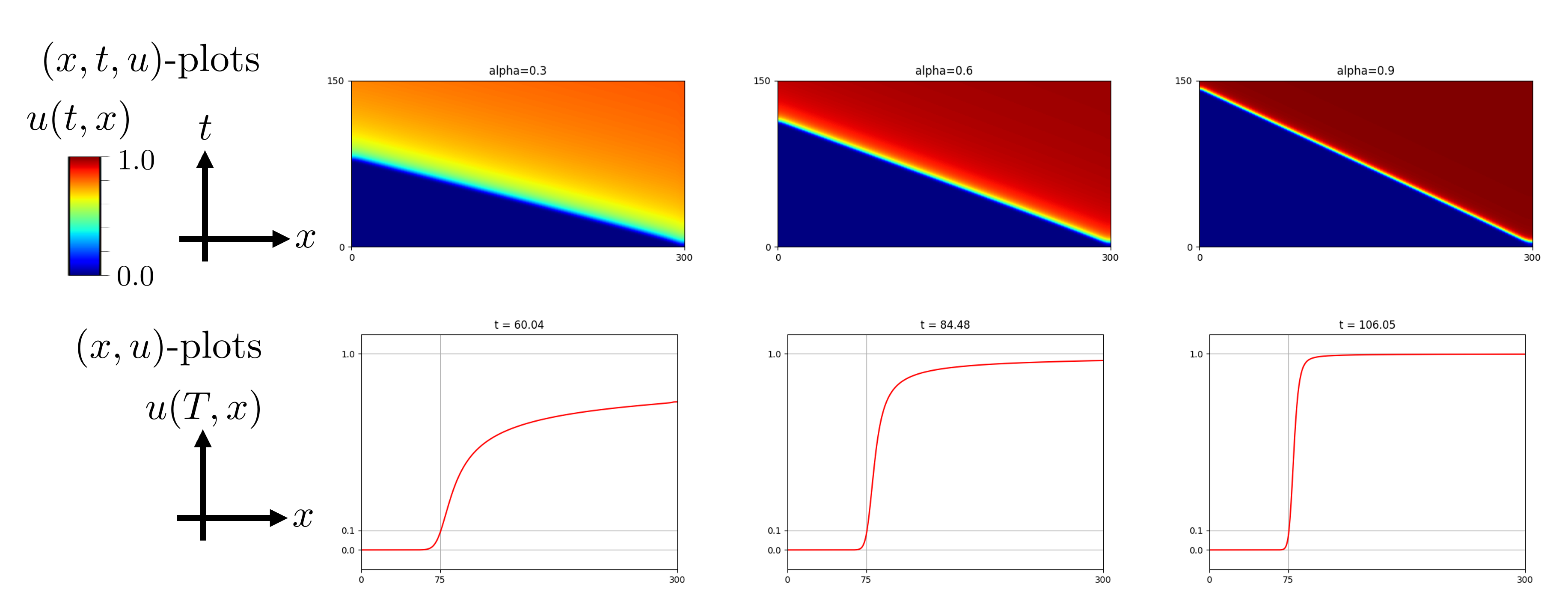}
\end{center}
\caption{\small{
Spatio-temporal plots (top panels) and profiles of the numerical solutions (bottom panels).
The parameters used are $l=300,\ l_0 = 295,\ \omega=0.01$, and $x_0 = 75$.
The numerical simulations correspond to $\alpha=0.3$ (left), $\alpha=0.6$ (middle), and $\alpha=0.9$ (right).
The solution profiles are shown at $T=60.04$ (left), $T=84.48$ (middle), and $T=106.05$ (right).
See Subsection \ref{subsec:num} for how to determine $T$.
}}
\label{fig:pro}
\end{figure}

Next, we discuss the propagation speed of the numerical solution.
In the case of the Fisher-KPP equation (i.e. $\alpha=1$),
it is known that when the initial state has a compact support or is a Heaviside function,
the propagation speed of the solution is characterized by a minimal speed at which a traveling wave solution with a monotone positive profile exists \cite{AW, KPP}.
In KPP-type nonlinearity, the minimal speed for traveling waves is equal to the minimal linear speed, $c^{*}_1=2\sqrt{f'(0)}$.
Here, the minimal linear speed is the minimal speed at which a positive solution $\phi(\xi)=e^{\lambda \xi}\ (\lambda>0)$ exists in the linearized equation
\beaa
c\phi'(\xi)= \phi''(\xi) + f'(0) \phi(\xi)\quad (\xi\in\bR)
\eeaa
around the unstable state $u=0$.
Then, $\lambda>0$ must satisfy $\lambda^2 -c\lambda + f'(0)=0$,
and it thus has a positive root $\lambda$ if and only if $c\ge c^{*}_1$.
In general, the property that the minimal speed of a traveling wave coincides with the minimal linear speed is called linear determinacy,
and it has been confirmed in the propagation phenomena of various reaction-diffusion systems \cite{AW, CLW, KPP}.


\medskip
Let us assume that a similar hypothesis holds for the behavior of the solution of equation \eqref{eq:main} by replacing the traveling wave solution with an asymptotic traveling wave solution.
In other words, the propagation speed of the solution is regarded as being characterized by the minimal speed for the asymptotic traveling wave solution.
We have not yet obtained results on the minimal speed, but we know from Theorem \ref{thm:main} that there exists at least an asymptotic traveling wave solution with $c\ge c^{*}_{\alpha}$.
Moreover, from Lemma \ref{lem:poly},
we have the linear minimal speed $c^{*}_{\alpha}$.
Therefore, if linear determinacy holds, the propagation speed is expected to be characterized by $c^{*}_{\alpha}$.


\medskip
A graph of $c^{*}_{\alpha}$ and the propagation speed of the numerical solution $c^{num}_{\alpha}$ is shown in Fig. \ref{fig:speed} (a).
It is seen that $c^{num}_{\alpha}$ is close to $c^{*}_{\alpha}$.
The relative errors are shown in Fig. \ref{fig:speed} (b).
The error is less than 0.1 and comparatively small in the range where $\alpha$ is close to 1.
This result suggests that $c^{*}_{\alpha}$ corresponds to the propagation speed of the solution in some sense.

\begin{figure}[bt]
\begin{center}
	\includegraphics[width=10cm]{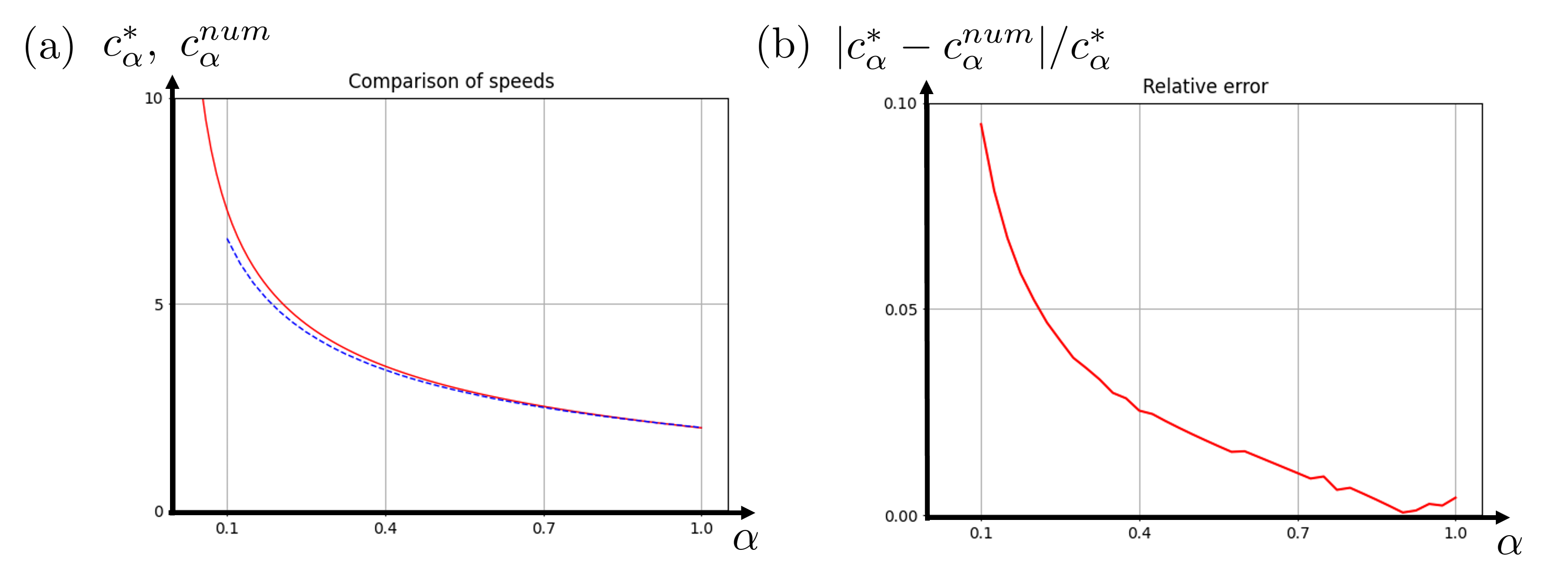}
\end{center}
\caption{\small{
(a) Comparison of $c^{*}_{\alpha}$ (red line) and the propagation speed of the numerical solution (dashed line).
The parameters are $l=300,\ l_0=295,\ \omega=0.01$, and $x_0 = 75$.
The blue dashed line is constructed by connecting straight line segments computed for $\alpha=0.1 + 0.025\times k\ (0\le k \le 36)$.
(b) Relative errors.}}
\label{fig:speed}
\end{figure}

\bigskip
\section{Discussion}
In this paper, we discussed the effect of the time-fractional derivative on front solution propagation.
By introducing a monostable nonlinear term to the time-fractional diffusion equation, we demonstrated that the sub-diffusion property is not conserved with respect to the time order of propagation.
Specifically, we introduced the concept of an asymptotic traveling wave solution to characterize solution propagation.
Through the analysis of the asymptotic traveling wave solution, we gained insights into the dependence of $\alpha$ on the propagation speed and the shape of the solution, and thus gained an understanding the effect of the time-fractional derivative.

Although the asymptotic traveling wave solution appears essential object for analyzing solution propagation,
we have not yet constructed a solution that converges to this solution.
One challenge in proving convergence lies in determining the appropriate initial state.
As the asymptotic traveling wave solution is not a solution to equation \eqref{eq:main},
estimating the initial conditions for convergence backward is not straightforward.
A useful piece of information is that the asymptotic traveling wave solution serves as a sub-solution.
The existence of a converging solution might be established if a well-constructed super-solution is introduced.

Several issues remain unresolved regarding the asymptotic traveling wave solution.
One such issue is the solution's asymptotic behavior as the solution approaches a stable equilibrium.
As mentioned in Section 1, convergence to a stable equilibrium at an algebraic decay rate is anticipated in the case of time-fractional derivatives.
In a related work \cite{NVN}, although for a different problem setting,
a bistable discontinuous nonlinear term was used to provide a concrete representation of the traveling wave solution and to construct a solution that decays algebraically to a constant state.
It might be feasible to achieve asymptotic behavior that decays algebraically by adopting a constructive method, even in this problem setting, but this remains an open question for future investigation.

Finally, research related to sub-diffusion and proliferation/saturation effects is discussed.
This includes studies on proliferation models with random walks, which differ from normal diffusion \cite{Fedotov, FSSS, FSSS2}.
The time order of the propagation speed of the solution has been investigated,
and was found to be {\red $O(t^{\alpha-1})$} for the model in \cite{FSSS2} and $O(t)$ for the model in \cite{Fedotov, FSSS}.
These differences arise from the nonlocality in the time direction that describes the effect of sub-diffusion.
As kinetics have the effect of propagating the solution in $O(t)$,
it is necessary to model the time nonlocality more strongly to propagate the solution in $O(t^{\alpha})$.
The finding that no sub-diffusion properties remain for the time order of the propagation in our setup is important information from a modeling point of view.
A mathematically rigorous approach to incorporating time non-locality into the proliferation model for sub-diffusion behavior would provide insights for modeling.

\bigskip
\section*{Acknowledgments}
The author expresses his sincere gratitude to Yikan Liu (Kyoto University)  for valuable advice, the referees for their valuable comments and suggestions for improving the paper, and Glenn Pennycook, MSc, from Edanz (https://jp.edanz.com/ac) for editing a draft of this manuscript.
This work was supported by JSPS KAKENHI Grant Number JP23K13013 and JP24H00188.


\end{document}